\documentclass[12pt,reqno]{amsart}
\usepackage{amsmath,amssymb,verbatim,comment,color}
\usepackage[pagebackref,pdftex,hyperindex]{hyperref}
\usepackage{graphicx}

\usepackage[margin=3cm]{geometry}
\usepackage{color}

\newcommand{\R}{\mathbb{R}}
\renewcommand{\S}{\mathbb{S}}
\newcommand{\T}{\mathbb{T}}

\newcommand{\Z}{\mathbb{Z}}

\newcommand{\cA}{\mathcal{A}}

\newcommand{\cL}{\mathcal{L}}

\newcommand{\cO}{\mathcal{O}}

\newcommand{\cT}{\mathcal{T}}

\renewcommand{\a}{\alpha}
\renewcommand{\b}{\beta}
\newcommand{\g}{\gamma}
\renewcommand{\d}{\delta}
\newcommand{\e}{\varepsilon}
\newcommand{\s}{\sigma}

\newcommand{\na}{\nabla}
\renewcommand{\r}{\rho}

\newcommand{\va}{\textbf{a}}

\newcommand{\p}{\partial}

\newcommand{\ddt}{\frac{d}{dt}}

\newcommand{\cf}{{\rm cf.\ }} 
 
\newcommand{\ie}{{\rm i.e.\ }} 
\newcommand{\etc}{{\rm etc.\ }}

\renewcommand{\bar}{\overline}

\DeclareMathOperator{\inj}{inj}

\DeclareMathOperator{\Rm}{Rm}
\DeclareMathOperator{\bRm}{\overline{Rm}}

\DeclareMathOperator{\SU}{SU}
\DeclareMathOperator{\Sp}{Sp}

\DeclareMathOperator{\Vol}{Vol}

\renewcommand{\leq}{\leqslant}
\renewcommand{\geq}{\geqslant}

\renewcommand{\hat}{\widehat}
\renewcommand{\tilde}{\widetilde}

\numberwithin{equation}{section}       

\newtheorem{prop} {Proposition} [section]
\newtheorem{thm}[prop] {Theorem} 
\newtheorem{dfn}[prop] {Definition}
\newtheorem{lem}[prop] {Lemma}
\newtheorem{cor}[prop]{Corollary}

\newtheorem{clm}[prop]{Claim}
\theoremstyle{remark}
\newtheorem*{ackn}{\bf{Acknowledgment}} 
\newtheorem{rk}[prop]{Remark}

\title[Convergence of mean curvature flow in hyperk\"ahler manifolds]{Convergence of mean curvature flow \\ in hyperk\"ahler manifolds} 
\date{\today}

 \author[K. Kunikawa]{Keita Kunikawa}
 \address{Advanced Institute for Materials Research\\
 Tohoku University\\
 Sendai 980-8577\\
 JAPAN}
 \email{keita.kunikawa.e2@tohoku.ac.jp}

 \author[R. Takahashi]{Ryosuke Takahashi}
 \address{Research Institute for Mathematical Siences\\
 Kyoto University\\
 Kyoto 606-8502\\
 JAPAN}
 \email{tryosuke@kurims.kyoto-u.ac.jp}
 
\subjclass[2010]{Primary: 53C44, Secondary: 53C26}
\keywords{mean curvature flow, hyperk\"ahler manifolds, hyper-Lagrangian submanifolds}
\begin{document}

\maketitle
Inspired by the work of Leung-Wan \cite{LW07}, we study the mean curvature flow in hyperk\"ahler manifolds starting from hyper-Lagrangian submanifolds, a class of middle dimensional submanifolds, which contains the class of complex Lagrangian submanifolds. For each hyper-Lagrangian submanifold, we define a new energy concept called the {\it twistor energy} by means of the associated twistor family (\ie 2-sphere of complex structures). We will show that the mean curvature flow starting at any hyper-Lagrangian submanifold with sufficiently small twistor energy will exist for all time and converge to a complex Lagrangian submanifold for one of the hyperk\"ahler complex structure. In particular, our result implies some kind of energy gap theorem for hyperk\"ahler manifolds which have no complex Lagrangian submanifolds.
\section{Introduction}
Let $(M,\bar{g})$ be a hyperk\"ahler $4n$-manifold, \ie the holonomy group is contained in $\Sp(n)$. Or equivalently, there exist distinct, $\bar{g}$-compatible complex structures $\{J_d\}_{d=1,2,3}$ which satisfy the {\it quaternion relations}:
\[
J_1^2=J_2^2=J_3^2=J_1 J_2 J_3=-{\rm Id}.
\]
Then each hyperk\"ahler manifold $M$ admits a 2-sphere of complex structures called the {\it twistor family}
\[
\sum_d c_d J_d \quad \text{for $(c_1,c_2,c_3) \in \S^2\subset \R^3$}.
\]
Throughout this paper, we assume that $(M, \bar{g})$ has bounded geometry (\ie the injectivity radius, curvatures and derivatives of the curvatures are uniformly bounded). Typical examples of hyperk\"ahler manifolds are a K3 surface and a compact torus $\T^4$ (In fact, any Calabi-Yau 4-manifold is hyperk\"ahler since $\SU(2) \simeq \Sp(1)$ and these are only compact 4-dimensional examples). Beauville \cite{Bea83} constructed two distinct deformation classes of hyperk\"ahler's in $4n$-dimension for every $n>1$. Moreover, Grady (\cf \cite{Gra99}, \cite{Gra03}) constructed two additional deformation classes in dimensions 12 and 20. Each deformation class has representatives which are moduli spaces of semistable sheaves on projective K3 surfaces or abelian surfaces or modifications of such moduli spaces.

In this paper, we show the existence and convergence result for the mean curvature flow (MCF) in hyperk\"ahler manifolds when the initial data is very small. It is no doubt that for studying the MCF, Lagrangian is one of the good class of submanifolds in a K\"ahler-Einstein manifold. Indeed, from Smoczyk's result \cite{Smo96}, the Lagrangian property is preserved under the MCF, and it gives a lot of benefits for computations of evolution equations, by identifying the extrinsic normal bundle with the intrinsic tangent bundle via the complex structure. Nevertheless, we would like to consider another class of submanifolds, called ``hyper-Lagrangian submanifolds'' as displayed below. This class includes Lagrangian submanifolds in hyperk\"ahler 4-manifolds. 

\subsection{Main result}
A natural counterpart of the Lagrangian condition in hyperk\"ahler manifolds is the ``complex Lagrangian'': for $J \in \S^2$, let $\Omega_J$ be a holomorphic symplectic form (\ie non-degenerate $J$-holomorphic 2-form)  with respect to $J$. For a $2n$-dimensional real submanifold $L \subset M$, we say that $L$ is {\it complex Lagrangian} if $\Omega_J|_L=0$ for some $J \in \S^2$. From a basic fact of hyperk\"ahler geometry, we find that there exists a $J$-orthogonal element $K \in \S^2$ such that $\Omega_J$ can be expressed as
\[
\Omega_J=\bar{\omega}_{JK}-\sqrt{-1} \bar{\omega}_K,
\]
where $\bar{\omega}_{JK}=\bar{g}(JK \cdot, \cdot)$, $\bar{\omega}_K=\bar{g}(K \cdot, \cdot)$ are real symplectic forms for $JK$ and $K$ respectively. So the condition $\Omega_J|_L=0$ means that two symplectic forms $\bar{\omega}_{JK}$ and $\bar{\omega}_K$ vanish at the same time for any $J$-orthogonal $K \in \S^2$.

However, this ``bi-Lagrangian'' condition is so strong that any complex Lagrangian submanifold $L$ in $M$ automatically becomes a (minimal) complex submanifold (\cf \cite{Hit99}). So, following the idea of Leung-Wan \cite{LW07}, we relax the assumption by using rich geometry on $M$. We say that $L$ is {\it hyper-Lagrangian} if $\Omega_{\Psi(x)}|_L=0$ at every point $x \in L$ for some varying complex structure $\Psi \colon L \to \S^2$. Then this map $\Psi$ is called the {\it complex phase}. In particular, complex Lagrangian is a special case when we can take $\Psi$ as a constant map. In \cite{LW07}, they showed that if the initial submanifold $L_0$ is hyper-Lagrangian, then $L_t:=F_t(L)$ is still hyper-Lagrangian under the MCF $F_t \colon L \to M$, and then the complex phase $\Psi_t$ evolves according to the coupled flow:
\begin{equation} \label{HLMCF}
\begin{cases}
\ddt F_t=H_t \\
\ddt \Psi_t=\Delta_t \Psi_t,
\end{cases}
\end{equation}
where $\Delta_t \Psi_t$ denotes the tension field of $\Psi_t$ with respect to the evolving metric $g_t:= F_t^{\ast}\bar{g}$. We would like to call \eqref{HLMCF} the {\it hyper-Lagrangian mean curvature flow} (HLMCF). Like other success stories of coupled flows (\cf \cite{Mul10}, \cite{Smo01}), the two geometric flows \eqref{HLMCF} can interact with each other to reveal better properties than it had by itself. For any hyper-Lagrangian submanifold $F \colon L \to M$, we introduce the {\it twistor energy} of $L$ as the Dirichlet energy of the complex phase $\Psi$ w.r.t. the induced metric $g:=F^{\ast}\bar{g}$:
\[
\cT(L):=\int_{L} |\nabla \Psi|^2 d \mu,
\]
where $d\mu$ denotes the Riemannian volume of $g$. Intuitively, the twister energy measures the deviation from $L$ being complex Lagrangian. We can show that any hyper-Lagrangian submanifold which is ``almost'' complex Lagrangian can be deformed to a genuine one in the following sense:
\begin{thm}[Convergence of the HLMCF] \label{coHLMCF}
Let $(M, \bar{g})$ be a hyperk\"ahler $4n$-manifold with bounded geometry. Suppose $L$ is a hyper-Lagrangian submanifold with the complex phase $\Psi_0$ which is smoothly immersed into $M$. Then for any $V_0$, $\Lambda_0$ and $\d_0>0$, there exists $\e_0=\e_0(n,V_0,\Lambda_0,\d_0, \bRm, \inj(M))>0$ such that if $L$ satisfies
\[
\Vol(L_0) \leq V_0, \quad |A|(0) \leq \Lambda_0, \quad \lambda_1(\Delta_L)(0) \geq \d_0, \quad \cT(L_0) \leq \e_0,
\]
then the hyper-Lagrangian mean curvature flow \eqref{HLMCF} starting from $L$ converges smoothly, exponentially fast to a complex Lagrangian submanifold in $M$ for one of the hyperk\"ahler complex structure on $M$.
\end{thm}
In the above theorem, we need not assume that $M$ has a complex Lagrangian submanifold, so it also gives an existence result for such a submanifold as well as the stability along the MCF. Although generic K3 surfaces do not have holomorphic curves at all, it is also interesting to understand this situation from geometric analytic point of view. Applying our theorem, one can immediately see that the twistor energy causes some gap: for any $V_0$, $\Lambda_0$ and $\d_0>0$ we define
\begin{align*}
\cL(V_0,\Lambda_0,\d_0):=
\Bigg\{
L\subset M
\Bigg|
\begin{split}
\text{$L$ is a hyper-Lagrangian submanifold} \\
 \; \Vol(L) \leq V_0, \quad |A| \leq \Lambda_0, \quad \lambda_1(\Delta_L) \geq \d_0
 \end{split}
\Bigg\}. 
\end{align*}
Then we have the following:
\begin{cor}[Energy gap theorem] \label{egt}
Assume that a $4n$-dimensional hyperk\"ahler manifold $M$ with bounded geometry has no complex Lagrangian submanifolds. Then for any $V_0$, $\Lambda_0$ and $\d_0>0$, there exists a constant $c=c(n, V_0, \Lambda_0, \d_0, \bRm, \inj(M))>0$ such that
\[
\inf_{L \in \cL(V_0,\Lambda_0,\d_0)} \cT(L) \geq c.
\]
\end{cor}
The proof of Theorem \ref{coHLMCF} is based on \cite{Li12} for the Lagrangian mean curvature flow (LMCF). In \cite{Li12}, the crucial step is to establish the exponential estimate for the $L^2$-norm of mean curvature vector $H$ by using the fact that each submanifold $L_t$ is Lagrangian, which is not valid for our case. Instead, we take an alternative approach from the view point of the theory of harmonic map flow. A key observation is that the $L^2$-norm of $H$ is bounded by the twistor energy (\cf Proposition \ref{hdp}):
\[
\int_{L_t} |H_t|^2 d \mu_t \leq 2 \cT(L_t).
\]
So the problem comes down to establishing the exponential estimate for the twistor energy, which is indeed, possible along the same line as the usual harmonic map flow (\cf Lemma \ref{efa}). We remark that for the harmonic map flow into positively curved targets, the flow possibly forms singularities in a finite time even if it has small initial Dirichlet energy (\cf \cite{CD90}). We can overcome this by showing Proposition \ref{bja}.

\subsection{Examples and relation to other results}
This paper is entirely written for hyper-Lagrangian submanifolds of arbitrary dimension. But after we posted the preprint, we noticed Qiu-Sun’s result \cite{QS19} which states that every hyper-Lagrangian except surface must be a complex Lagrangian, so the concept of the hyper-Lagrangian is meaningful only when $n=1$. However, we emphasize that our results are new even when $n=1$. Contrary to the higher dimensional case, the concept of hyper-Lagrangian surface is universal and enables us to make a systematic study of several conditions for submanifolds preserved under the MCF. We can see that every surface $L$ in a hyperk\"ahler 4-manifold $M$ admits a canonical complex phase map $\Psi \colon L \to \S^2$ defined by
\[
J_{\Psi} e_1=e_2, \quad J_{\Psi} e_3=-e_4,
\]
where $\{e_1,e_2,e_3,e_4 \}$ is any oriented orthonormal frame on $TM$ such that $\{e_1,e_2 \}$ is an oriented frame on $TL$ and $\{e_3,e_4\}$ is an orthonormal frame for the normal bundle. Indeed, the map $\Psi$ is independent of the choice for such a frame. In the following, we will explain the each class of submanifolds separately while considering what shape the each complex phase is (see also \cite{LW07}).
\subsubsection{Symplectic mean curvature flow}
First, we consider symplectic surfaces. It was asked by Yau (for instance, see \cite{Wan01}) that {\it how can a symplectic submanifold be deformed to a holomorphic one?} Since a symplectic surface remains to be symplectic along the MCF in a K\"ahler-Einstein surface (\cf \cite{CL01}, \cite{Wan01}), one expects that the symplectic mean curvature flow (SMCF) is applicable to Yau's question. It seems that the convergence of the SMCF with small initial data has not been accomplished yet in the general case, whereas we know several partial results. For instance, our theorem generalizes Han-Sun's result \cite[Corollary 4.6]{HS12}: we express $\Psi$ as a map $\va \colon L \to \R^3$, \ie $\va$ is a coefficient of $\Psi$ with respect to $\{J_d \}$
\[
J_{\Psi}=\sum_d a_d J_d, \quad \va:=(a_1,a_2,a_3).
\]
By using the quaternion relations, we see that
\begin{equation} \label{kag}
\cos \a:=\bar{\omega}_{J_3}(e_1,e_2)=\bar{g}(J_3 e_1,e_2)=a_3.
\end{equation}
Hence the condition that $L$ is symplectic w.r.t. $\bar{\omega}_{J_3}$ is equivalent to say that the image $\Psi(L)$ is contained in the hemisphere $\S^2_+:=\{(c_1,c_2,c_3) \in \S^2 \subset \R^3|c_3>0\}$. Then the (local) angle $\a$ defined by \eqref{kag} is called the {\it k\"ahler angle}. Applying the maximum principle to the evolution equation of $\va$, we find that the hemisphere condition is preserved under the HLMCF (\cf Corollary \ref{mpl}), which is essentially a restatement of the fact as explained above that {\it if the initial surface is symplectic, then the surface is still symplectic along the mean curvature flow}. In \cite{HS12}, they showed the convergence of the SMCF under the stronger assumption that the ambient K\"ahler surface $M$ has zero sectional curvature and the initial $L^2$-norm of $A$ is very small. Also there is a convergence result for the SMCF in K\"ahler-Einstein surfaces with positive Ricci curvature by Han-Li \cite{HL05}, where the positivity of the extrinsic curvature was essentially used. Anyways, Theorem \ref{coHLMCF} indicates that the MCF method is still valid for Yau's question, and makes the first step in this direction.
\subsubsection{Lagrangian mean curvature flow}
Next, we explain the Lagrangian case. If $L$ is Lagrangian with respect to $\bar{\omega}_{J_o}$ for a fixed $J_o \in \S^2$, then without loss of generality, we may assume $J_3=J_o$. By the Lagrangian condition, we find that $L$ has the $J_3$-orthogonal complex phase $J_{\Psi}$ which can be expressed as
\begin{equation} \label{Lag}
J_{\Psi}(x)=\cos \theta(x) J_1+\sin \theta(x)J_2
\end{equation}
for some multi-valued function $\theta \colon L \to \R$. Moreover, the function $\theta$ and $\bar{\omega}_{J_o}$ are related by the formula
\begin{equation} \label{exact}
i_H \bar{\omega}_{J_o}=d \theta.
\end{equation}
So $\theta$ is nothing but the {\it Lagrangian angle}. In particular, we often consider the following special cases:
\begin{enumerate}
\item The form $i_H \bar{\omega}_{J_o}$ is exact, or equivalently, $\theta$ is a single-valued function.
\item The submanifold $L$ is {\it almost calibrated}, \ie $L$ satisfies (1) and $\cos \theta >0$.
\end{enumerate}
As it is for Lagrangian, these two conditions are preserved under the MCF (\cf \cite{Smo99}, \cite{CL01}, \cite{Wan01}). The convergence result for the LMCF with small initial data was obtained by Li \cite[Theorem 1.2]{Li12}. He showed the similar convergence result to Theorem \ref{coHLMCF} under the assumption (1) (but, we need not assume (2)) and that the initial $L^2$-norm of $H$ is very small. So Theorem \ref{coHLMCF} is still meaningful even if $L_0$ is Lagrangian since we need not assume (1) in our theorem. 

\begin{figure}[h]\label{concept}
\includegraphics[width=10.0cm]{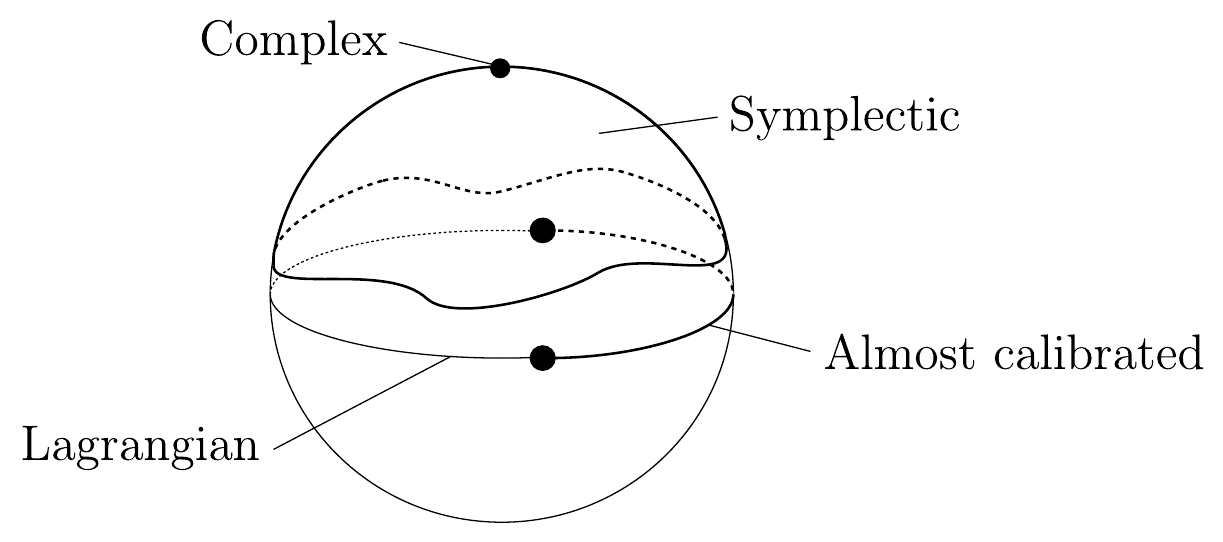}
\caption{Image of the complex phase $\Psi$ in $\S^2$}
\end{figure}

Finally, we again emphasize the benefit of the hyper-Lagrangian submanifolds. In fact, the hyper-Lagrangian structure gives one a comprehensive view point to understand the concepts of symplectic surfaces or (almost calibrated) Lagrangian submanifolds in hyperk\"ahler 4-manifolds. Figure 1 shows the correspondence between each of these concepts and the image of the complex phase map $\Psi:L\to \S^2$. 
\subsubsection{Holomorphic curves in K3}
On any polarized K3 surface $(M,H)$ (with $H \not \simeq \cO_M$), it is known that there exists at least one holomorphic curve which belongs to the linear system $|mH|$ for all $m \geq 1$ (Bogomolov, Mumford, Mori-Mukai \cite{MM83}). Due to the Lefschetz theorem, the existence of such an $H$ is equivalent to say that the {\it N\'eron-Severi lattice}
\[
{\rm NS}(M):=H^{1,1}(M) \cap H^2(M,\Z)
\]
is non-empty. Moreover, Chen \cite{Che99} proved the existence of infinitely many holomorphic curves on general K3 surfaces. Then we can take any small perturbation of the holomorphic curves as an initial data in Theorem \ref{coHLMCF}.
\subsection{Organization of the paper}
Our article will be organized as follows. We will first recall some results discovered by Leung-Wan \cite{LW07} and prove formulas relating the mean curvature vector (or second fundamental form) with the complex phase which are needed in the rest of the article. In Section 3, we study the behavior of the twistor energy and first eigenvalue along the HLMCF, and then establish some parabolic estimates. Finally, we give the proof of Theorem \ref{coHLMCF} in the last part of Section 3.
\begin{ackn}
The authors express their gratitude to Shigetoshi Bando for helpful conversations. We also would thank Ryoichi Kobayashi for pointing out Corollary \ref{egt}. K.K. is supported by JSPS KAKENHI Grant Number JP19K14521, and R.T. is supported by Grant-in-Aid for JSPS Fellows Number 16J01211.
\end{ackn}
\section{Hyper-Lagrangian submanifolds}
In this section, we recall some results about hyper-Lagrangian submanifolds studied in \cite{LW07}. Let $M$ be a hyperk\"ahler $4n$-manifold and $L \subset M$ a real submanifold of dimension $2n$. In this section, we promise that the indices ($i$, $j$, $\a$, $\b$, \etc) run in the following manner
\[
\begin{split}
i,j=1,\ldots,2n, \quad \a,\b=2n+1,\ldots,4n, \quad A,B=1,\ldots,4n, \\
\nu, \lambda=1,\ldots n, \quad \mu, \r=n+1,\ldots,2n.
\end{split}
\]
\begin{dfn}
A submanifold $L$ is called hyper-Lagrangian if $\Omega_{\Psi(x)}|_L=0$ at every point $x \in L$ for some $\Psi \colon L \to \S^2$. Then $\Psi$ is called the complex phase. In particular, a hyper-Lagrangian submanifold is called complex Lagrangian if we can take $\Psi$ as a constant map. 
\end{dfn}
Let $\Phi \colon L \to \S^2$ be a smooth map such that $\Phi(x)$ is orthogonal to $\Psi(x)$ for each $x \in L$. We can take a special orthonormal frame $\{e_i\}$ for $TL$ satisfying
\[
J_{\Psi} e_{2 \nu-1}=e_{2 \nu}.
\]
Then $\{e_{i+2n}:=J_{\Phi} e_i\}$ is an orthonormal frame for the normal bundle satisfying
\[
J_{\Psi} e_{2 \mu-1}=-e_{2 \mu}.
\]
Then $\{e_A\}$ defines a frame of $TM$. For a hyper-Lagrangian submanifold $L$ with the complex phase $\Psi$, we denote the associated almost-complex structure by $J_{\Psi}$. Then the complex phase $J_{\Psi}$ acts on $TL$, and determines an almost-complex structure on $L$. However, hyper-Lagrangian is a strong condition which imposes a lot of restrictions on the structural equations. For instance, let $\{\varphi_{AB}\}$ be the connection forms with respect to $\{e_A\}$, \ie $\bar{\na} e_A=\varphi_{AB} e_B$. Then the structure theorem of hyper-Lagrangian submanifolds (\cf \cite[Theorem 4.1]{LW07}) implies
\begin{equation} \label{seq}
\begin{split}
\varphi_{2\nu-1,2\lambda-1}=\varphi_{2\nu,2\lambda}, \quad \varphi_{2\nu,2\lambda-1}=-\varphi_{2\nu-1,2\lambda}, \\
\varphi_{2\mu-1,2\r-1}=-\varphi_{2\mu,2\r}, \quad \varphi_{2\mu,2\r-1}=\varphi_{2\mu-1,2\r}.
\end{split}
\end{equation}
As a consequence, we obtain the following:
\begin{thm}[\cite{LW07}, Corollary 4.2]
The complex phase $\Psi$ induces an integrable K\"ahler structure $(J_\Psi,\bar{g}|_L)$ on $L$ with holomorphic normal bundle.
\end{thm}
We set
\[
e_{\nu}'=\frac{1}{2}(e_{2 \nu-1}-\sqrt{-1} e_{2 \nu}), \quad e_{\nu}''=\frac{1}{2}(e_{2 \nu-1}+\sqrt{-1} e_{2 \nu}),
\]
\[
e_{\mu}'=\frac{1}{2}(e_{2 \mu-1}+\sqrt{-1} e_{2 \mu}), \quad e_{\mu}''=\frac{1}{2}(e_{2 \mu-1}-\sqrt{-1} e_{2 \mu}).
\]
Then $\{e_{\nu}',e_{\mu}' \}$ defines a complex basis referred as the {\it canonical frame adapted to $(\Psi,\Phi)$}.
Correspondingly, we take the basis $\{\zeta_A\}$ dual to $\{e_A\}$ and set
\[
\zeta_{\nu}'=\zeta_{2 \nu-1}+\sqrt{-1} \zeta_{2 \nu}, \quad \zeta_{\nu}''=\zeta_{2 \nu-1}-\sqrt{-1} \zeta_{2 \nu},
\]
\[
\zeta_{\mu}'=\zeta_{2 \mu-1}-\sqrt{-1} \zeta_{2 \mu}, \quad \zeta_{\mu}''=\zeta_{2 \mu-1}+\sqrt{-1} \zeta_{2 \mu}.
\]
With this basis, $\Omega_\Psi$ can be written by
\[
\Omega_\Psi=-\sqrt{-1} \sum_{\nu,\mu} \zeta_{\nu}' \wedge \zeta_{\mu}'.
\]
Leung-Wan (\cf \cite[Theorem 4.5]{LW07}) found the formula relating the mean curvature vector $H$ and the complex phase $\Psi$ as follows:
\begin{prop}
We have
\begin{equation} \label{plf}
i_H \Omega_\Psi+2 \sqrt{-1} \p \Psi=0.
\end{equation}
\end{prop}
In particular, the above proposition shows that a hyper-Lagrangian submanifold $L$ is minimal if and only if the complex phase $\Psi$ is anti-holomorphic. Meanwhile, by using the formula \eqref{plf}, one can obtain a bound for $|H|$ by means of the energy density of the complex phase $\Psi$:
\begin{prop} \label{hdp}
We have
\[
|H|^2 \leq 2|\na \Psi|^2.
\]
\end{prop}
\begin{proof}
For a fixed $x \in L$, we set
\[
J'_1=J_{\Psi}(x), \quad J'_2=J_\Phi (x), \quad J'_3=J'_1 J'_2.
\]
We would like to call it the {\it canonical basis adapted to $(\Psi,\Phi)$ at $x$}. Then we set the coefficient $\va'=(a'_1,a'_2,a'_3)$ as $J_\Psi=\sum_d a'_d J'_d$. We take a local representation of $\Psi$:
\[
\Theta(p)=\frac{a'_1(p)+\sqrt{-1} a'_2(p)}{1-a'_3(p)}
\]
via stereographic projection. Then the formula \eqref{plf} yields that
\[
i_H \Omega_\Psi+2 \sqrt{-1} \p \Theta=0 \quad \text{at $x$}.
\]
From the construction, we know that
\[
a'_1(x)=1, \quad a'_2(x)=a'_3(x)=0, \quad \Theta(x)=1.
\]
Also since $L$ is hyper-Lagrangian with the complex phase $\Psi$, the derivative $\bar{\na} J_\Psi$ is spanned by $J'_2$ and $J'_3$ at $x$, so
\[
da'_1|_x=0.
\]
Thus we have
\[
\p \Theta|_x=\sqrt{-1} \p a'_2|_x+\p a'_3|_x,
\]
\[
|\p \Theta|^2 \leq 2(|\p a'_2|^2+|\p a'_3|^2)=|da'_2|^2+|da'_3|^2=|\na \va'|^2
\]
at $x$. On the other hand, if we set $H=-\sum_\a H^{\a} e_{\a}$, one can easily observe that
\[
i_H \Omega_\Psi=-\sqrt{-1} \sum_\mu (H^{2\mu-1}-\sqrt{-1} H^{2 \mu}) \zeta_{\mu}',
\]
\[
|i_H \Omega_\Psi|^2=2|H|^2.
\]
So we have
\[
|H|^2=2|\p \Theta|^2 \leq 2|\na \va'|^2.
\]
We note that $\va'$ and $\Theta$ heavily depend on the choice of the basis $(J'_1,J'_2,J'_3)$ whereas $\va$ only depends on the background basis $(J_1,J_2,J_3)$. However, the point is that the norm $|\na \va'|^2$ is independent of the choice of an orthogonal basis $(J'_1,J'_2,J'_3)$ since the Euclidean metric on $\R^3$ is invariant under the standard $O(3)$-action. So we have $|\na \va'|=|\na \va|=|\na \Psi|$ and $|H|^2 \leq 2|\na \Psi|^2$.
\end{proof}
We also remark that the quantity $|\na \Psi|$ has the following three equivalent definitions:
\begin{itemize}
\item We regard the complex phase $\Psi$ as a map $\va \colon L \to \S^2 \subset \R^3$, and define $|\na \Psi|$ as the energy density of $\va$:
\[
|\na \va|^2=\sum_d |\na a_d|_g^2.
\]
\item We define $|\na \Psi|$ as the energy density of $\Psi \colon L \to \S^2$, \ie a map into $\S^2$ (also see \eqref{etd}).
\item We define $|\na \Psi|$ as the norm of the covariant derivative of $J_\Psi$ along $L$:
\[
|\bar{\na} J_{\Psi}|^2=\sum_{i,A,B} \bar{g}((\bar{\na}_i J)(e_A),e_B)^2,
\]
where $\bar{\na}$ denotes the Levi-Civita connection on the ambient space $(M,\bar{g})$. Then, taking account into the fact that $\{J_d\}$ is parallel and $\langle J_d, J_e \rangle_{\bar{g}}=4n \d_{de}$, we have $\bar{\na} J_{\Psi}=\sum_d da_d \otimes J_d$ and $|\bar{\na} J_{\Psi}|=2 \sqrt{n} |\na \va|$.
\end{itemize}
As for the relation to the second fundamental form $A$, we have the following:
\begin{prop} \label{bja}
In the canonical frame adapted to $(\Psi,\Phi)$, the quantity $|\bar{\na} J_{\Psi}|^2$ is expressed as
\[
|\bar{\na} J_{\Psi}|^2=4 \sum_{i,\nu,\mu} \big[(h^{2\mu-1}_{2\nu,i}-h^{2\mu}_{2\nu-1,i})^2+(h^{2\mu-1}_{2\nu-1,i}+h^{2\mu}_{2\nu,i})^2 \big],
\]
where $h^{\a}_{ij}:=\bar{g}(e_i, \bar{\na}_j e_{\a})$. In particular, we have
\[
|\na \Psi| \leq c(n)|A|.
\]
\end{prop}
\begin{proof}
Set $J_{i,A,B}:=\bar{g}(\bar{\na}_i J_\Psi(e_A),e_B)$ for simplicity. We compute
\begin{eqnarray*}
(\bar{\na}J_{\Psi})(e_{2\nu-1})&=&\bar{\na}(e_{2\nu})-J_\Psi(\bar{\na}e_{2\nu-1}) \\
&=& \sum_j \varphi_{2\nu,j}e_j+\sum_\a \varphi_{2\nu,\a}e_\a-J_\Psi \bigg( \sum_j \varphi_{2\nu-1,j}e_j+\sum_\a \varphi_{2\nu-1,\a}e_\a \bigg).
\end{eqnarray*}
By using \eqref{seq}, we know that the first and third terms cancel each other out. So we have
\[
(\bar{\na}J_{\Psi})(e_{2\nu-1})=\sum_{\mu} \big[ (\varphi_{2\nu,2\mu-1}-\varphi_{2\nu-1,2\mu})e_{2\mu-1}+(\varphi_{2\nu,2\mu}+\varphi_{2\nu-1,2\mu-1})e_{2\mu} \big],
\]
and hence
\[
J_{i,2\nu-1,j}=0, \quad J_{i,2\nu-1,2\mu-1}=-h^{2\mu-1}_{2\nu,i}+h^{2\mu}_{2\nu-1,i}, \quad J_{i,2\nu-1,2\mu}=-h^{2\mu}_{2\nu,i}-h^{2\mu-1}_{2\nu-1,i}.
\]
In the same way, we can compute other terms by using \eqref{seq} as follows:
\[
\begin{split}
J_{i,2\nu,j}=0, \quad J_{i,2\nu,2\mu-1}=h^{2\mu}_{2\nu,i}+h^{2\mu-1}_{2\nu-1,i}, \quad J_{i,2\nu,2\mu}=-h^{2\mu-1}_{2\nu,i}+h^{2\mu}_{2\nu-1,i}, \\
J_{i,2\mu-1,\a}=0, \quad J_{i,2\mu-1,2\nu-1}=h^{2\mu-1}_{2\nu,i}-h^{2\mu}_{2\nu-1,i}, \quad J_{i,2\mu-1,2\nu}=-h^{2\mu}_{2\nu,i}-h^{2\mu-1}_{2\nu-1,i}, \\
J_{i,2\mu,\a}=0, \quad J_{i,2\mu,2\nu-1}=h^{2\mu-1}_{2\nu-1,i}+h^{2\mu}_{2\nu,i}, \quad J_{i,2\mu,2\nu}=h^{2\mu-1}_{2\nu,i}-h^{2\mu}_{2\nu-1,i}.
\end{split}
\]
So we obtain the desired formula.
\end{proof}
\section{hyper-Lagrangian mean curvature flow}
\subsection{Evolution of the coefficient vector}
We regard the complex phase $\Psi$ as a map into $\S^2 \subset \R^3$ and write $\va=(a_1,a_2,a_3)$. We compute the evolution equation of $\va$ when $\Psi$ evolves along the generalized harmonic map flow $\ddt \Psi=\Delta_t \Psi$.
\begin{lem}
Along the HLMCF, $\va$ satisfies
\begin{equation} \label{evc}
\bigg(\ddt-\Delta_t\bigg)\va=|\na \va|^2 \va.
\end{equation}
\end{lem}
\begin{proof}
We take a polar coordinate  $(\theta, \varphi)$ of $\S^2$ and express $\va$ as
\[
\va=
\begin{pmatrix}
\cos \Psi^{\theta} \sin \Psi^{\varphi} \\
\sin \Psi^{\theta} \sin \Psi^{\varphi} \\
\cos\Psi^{\varphi}
\end{pmatrix},
\]
where we write $\Psi^{\theta}=\theta \circ \Psi$, $\Psi^{\varphi}=\varphi \circ \Psi$ for simplicity. Then
\[
\ddt \Psi=\ddt \Psi^{\theta} \cdot \frac{\p}{\p \theta} \circ \Psi+\ddt \Psi^{\varphi} \cdot \frac{\p}{\p \varphi} \circ \Psi.
\]
Let $(x^1,\ldots,x^{2n})$ be a local coordinate in $L$. Recall the definition of the tension field of $\Psi$:
\[
\Delta \Psi=\sum_{i,j=1}^n g^{i j} \hat{\na}_i \hat{\na}_j \Psi^{\theta} \cdot \frac{\p}{\p \theta} \circ \Psi+\sum_{i,j=1}^n g^{i j} \hat{\na}_i \hat{\na}_j \Psi^{\varphi} \cdot \frac{\p}{\p \varphi} \circ \Psi \in C^{\infty}(\Psi^{-1} T\S^2),
\]
where $\hat{\nabla}$ denotes the canonical connection on $\Psi^{-1} T\S^2$ associated to $g$ and the standard metric $\tilde{g}$ on $\S^2$. Then
\[
\hat{\na}_i \hat{\na}_j \Psi^{\a}=\na_i \na_j \Psi^{\a}+\sum_{\b,\g=\theta, \varphi} \tilde{\Gamma}^{\a}_{\b \g}(\Psi) \frac{\p \Psi^{\b}}{\p x^i} \cdot \frac{\p \Psi^{\g}}{\p x^j}, \quad \a=\theta, \varphi,
\]
where $\tilde{\Gamma}^{\a}_{\b \g}$ denotes the Christoffel symbol w.r.t. $\tilde{g}$. We can easily compute
\[
\tilde{g}_{\theta \theta}=\sin^2 \varphi, \quad \tilde{g}_{\theta \varphi}=0, \quad \tilde{g}_{\varphi \varphi}=1,
\]
\[
\tilde{\Gamma}^{\theta}_{\theta \theta}=\tilde{\Gamma}^{\varphi}_{\varphi \varphi}=0, \quad
\tilde{\Gamma}^{\theta}_{\theta \varphi}=\frac{\cos \varphi}{\sin \varphi}, \quad \tilde{\Gamma}^{\varphi}_{\theta \theta}=-\sin \varphi \cos \varphi.
\]
This implies that
\[
\ddt \Psi^{\theta}=\sum_{i,j=1}^n g^{i j} \hat{\na}_i \hat{\na}_j \Psi^{\theta}=\Delta \Psi^{\theta}+\frac{\cos \Psi^{\varphi}}{\sin \Psi^{\varphi}} \cdot \langle \na \Psi^{\theta}, \na \Psi^{\varphi} \rangle_g,
\]
\[
\ddt \Psi^{\varphi}=\sum_{i,j=1}^n g^{i j} \hat{\na}_i \hat{\na}_j \Psi^{\varphi}=\Delta \Psi^{\varphi}-\sin \Psi^{\varphi} \cos \Psi^{\varphi} \cdot |\na \Psi^{\theta}|_g^2.
\]
Since
\[
\Delta a_3=-\sin \Psi^{\varphi} \cdot \Delta \Psi^{\varphi}-\cos \Psi^{\varphi} \cdot |\nabla \Psi^{\varphi}|_g^2,
\]
\begin{equation} \label{etd}
|\na \va|^2=\sin^2 \Psi^{\varphi} \cdot |\na \Psi^{\theta}|_g^2+|\nabla \Psi^{\varphi}|_g^2,
\end{equation}
we have
\[
\ddt a_3=-\sin \Psi^{\varphi} \cdot \ddt \Psi^{\varphi}=\Delta a_3+|\na \va|^2 a_3.
\]
We can compute the evolution equation of $a_1$ and $a_2$ in the similar way.
\end{proof}
Applying the maximum principle to \eqref{evc}, we obtain
\begin{cor}[also see \cite{LW07}, Theorem 5.1] \label{mpl}
If $L_0$ satisfies $a_3>c$ for some constant $c \in (0,1)$ then $a_3>c$ holds along the HLMCF $L_t$ for all $t \in [0,T]$. In particular, the hemisphere condition $\Psi(L) \subset \S^2_+$ is preserved under the HLMCF.
\end{cor}
\subsection{$L^2$-estimates}
Let $L \subset M$ be a hyper-Lagrangian submanifold with the complex phase $\Psi$. 
\begin{dfn}
We define the twistor energy of $L$ as the Dirichlet energy of the complex phase:
\[
\cT(L):=\int_{L} |\nabla \Psi|^2 d \mu.
\]
\end{dfn}
By using \eqref{evc}, we can obtain the exponential estimate for the twistor energy:
\begin{lem}[Exponential estimate for the twistor energy] \label{efa}
For the HLMCF $L_t$, we have
\[
\ddt \cT(L_t) \leq (-2\lambda_1(t)+C(n) \max_{L_t} |H| |A|+2\max_{L_t}|\na \Psi|^2) \cdot \cT(L_t),
\]
where $\lambda_1(t)>0$ denotes the first eigenvalue of the Laplacian $\Delta_t$.
\end{lem}
\begin{proof}
First, we recall the evolution of the Riemannian metric on $L$ (for instance, see \cite{CL01}):
\[
\ddt g_{ij}=-2H^{\a} h_{ij}^{\a}.
\]
By using this and the expression of the energy density as the norm of the coefficient vector $|\na \Psi|^2=|\na \va|^2$, we compute
\[
\ddt \int_L |\na \va|^2 d \mu_t=2 \int_L \langle \na \ddt \va, \na \va \rangle d \mu_t+\int_L \sum_d \ddt g^{i j} \na_i a_d \na_j a_d d \mu_t-\int_L |\na \va|^2 |H|^2 d \mu_t.
\]
We estimate each term separately. The first term is
\begin{eqnarray*}
2 \int_L \langle \na \ddt \va, \na \va \rangle d \mu_t &=& 2 \int_L \langle \na ((\Delta+|\na \va|^2)\va), \na \va \rangle d \mu_t \\
&=& -2 \int_L |\Delta \va|^2 d \mu_t-2 \int_L |\na \va|^2 \langle \va, \Delta \va \rangle d \mu_t \\
&\leq& -2 \lambda_1 \int_L |\na \va|^2 d \mu_t+2 \int_L |\na \va|^4 d \mu_t \\
&\leq& -2 \lambda_1 \int_L |\na \va|^2 d \mu_t+2 \max_{L_t}|\nabla \va|^2 \int_L |\na \va|^2 d \mu_t,
\end{eqnarray*}
where we used the formula
\[
0=\langle \ddt \va, \va \rangle=\langle (\Delta+|\na \va|^2)\va, \va \rangle=\langle \Delta \va, \va \rangle+|\na \va|^2,
\]
which can be proved easily by differentiating $|\va|^2=1$ in $t$. For the second term, we have
\begin{eqnarray*}
\bigg| \int_L \sum_d \ddt g^{i j} \na_i a_d \na_j a_d d \mu_t \bigg| &=& \bigg|2 \int_L \sum_d H^{\a} h_{ij}^{\a} \na_i a_d \na_j a_d d \mu_t \bigg| \\
&\leq& C(n) \max_{L_t} |H| |A| \cdot \int_L |\na \va|^2 d \mu_t.
\end{eqnarray*}
This completes the proof of the Lemma.
\end{proof}
The above lemma says that we need to control $\lambda_1$ in order to obtain a bound for the twistor energy. So we establish the exponential estimate for $\lambda_1$ as follows:
\begin{lem}[Exponential estimate for the first eigenvalue] \label{efe}
Along the HLMCF, the first eigenvalue $\lambda_1(t)$ satisfies
\[
\ddt \lambda_1 \geq -(\max_{L_t}|H|^2+C(n) \max_{L_t} |H||A|) \cdot \lambda_1.
\]
\end{lem}
\begin{proof}
Let $f$ be an eigenfunction w.r.t. $\lambda_1$, \ie $f$ satisfies
\[
-\Delta_t f=\lambda_1 f, \quad \int_L f^2 d \mu_t=1.
\]
Then the first eigenvalue $\lambda_1$ is
\[
\lambda_1=\int_L |\na f|^2 d \mu_t.
\]
Differentiating $\int_L f^2 d \mu_t=1$ in $t$, we have
\[
\int_L \bigg(2\ddt f \cdot f-f^2 |H|^2\bigg) d \mu_t=0.
\]
Thus we can compute
\begin{eqnarray*}
\ddt \lambda_1 &=& 2 \int_L \Big\langle \na \ddt f, \na f \Big\rangle d \mu_t+\int_L \ddt g^{ij} \na_i f \na_j f d \mu_t-\int_L |\na f|^2 |H|^2 d \mu_t \\
&=& -2 \int_L \ddt f \cdot \Delta f d \mu_t+2 \int_L H^{\a} h_{ij}^{\a} \na_i f \na_j f d \mu_t+\int_L f \Delta f \cdot |H|^2 d \mu_t \\
&+& \int_L f \langle \na f, \na |H|^2 \rangle d \mu_t.
\end{eqnarray*}
Using the relation $-\Delta f=\lambda_1 f$, we find that the first term and the third term cancel each other out. The second term can be estimates as
\[
\bigg| 2 \int_L H^{\a} h_{ij}^{\a} \na_i f \na_j f d \mu_t \bigg| \leq C(n) \max_{L_t} |H||A| \cdot \lambda_1.
\]
The fourth term is
\begin{eqnarray*}
\int_L f \langle \na f, \na |H|^2 \rangle d \mu_t &=& - \int_L (f \Delta f+|\na f|^2) |H|^2 d \mu_t \\
&=& \lambda_1 \int_L f^2 |H|^2 d \mu_t-\int_L |\na f|^2 |H|^2 d \mu_t \\
&\geq&- \max_{L_t} |H|^2 \cdot \lambda_1.
\end{eqnarray*}
Thus we obtain the desired result.
\end{proof}
\subsection{$C^0$-estimates}
In order to get the $C^0$-estimates from the $L^2$, the notion of non-collapsing geodesic ball is convenient.
Roughly speaking, it says that the volume of each geodesic ball in $L$ is bounded from below by that of the Euclidean geodesic ball of the same radius. Let $N$ be a compact Riemannian $m$-manifold.
\begin{dfn}
We say that
\begin{enumerate}
\item A geodesic ball $B(x,\r)$ in $N$ is called $\kappa$-noncollapsed if
\[
\frac{\Vol(B(y,s))}{s^m} \geq \kappa
\]
holds whenever $B(y,s) \subset B(x,\r)$.
\item A compact Riemannian manifold $N$ is called $\kappa$-noncollapsed on the scale $r$ if every geodesic ball $B(x,s)$ is $\kappa$-noncollapsed for $s \leq r$.
\end{enumerate}
\end{dfn}
\begin{lem} \label{efz}
Let $(E,h,D)$ be a vector bundle with a fiber metric $h$ and a compatible connection $D$ over a compact Riemmanian manifold $N$. Assume that $N$ is $\kappa$-noncollapsed on the scale $r$. For any smooth section $\s \in C^{\infty}(E)$, if
\[
|D \s| \leq \Lambda, \quad \int_N |\s
|^2 d \mu \leq \e \leq r^{m+2},
\]
then
\[
\max_N |\s| \leq (\Lambda+\kappa^{-1/2}) \e^{\frac{1}{m+2}}.
\]
\end{lem}
\begin{proof}
Assume that $|\s|$ attains its maximum at a point $x_0 \in N$ and the statement does not hold, \ie
\[
|\s(x_0)| > (\Lambda+\kappa^{-1/2}) \e^{\frac{1}{m+2}}.
\]
Then by setting $\d:=\e^{\frac{1}{m+2}}$, we get
\[
\Lambda \d=\Lambda \e^{\frac{1}{m+2}}<|\s(x_0)|.
\]
Thus for any $x \in B(x_0, \d)$, we have
\[
|\s(x)| \geq |\s(x_0)|-\Lambda \d>0.
\]
Integrating on $B(x_0, \d)$ yields that
\[
\e \geq \int_{B(x_0, \d)} |\s|^2 d \mu \geq (|\s(x_0)|-\Lambda \d)^2 \Vol(B(x_0, \d)) \geq (|\s(x_0)|-\Lambda \d)^2 \kappa \d^m,
\]
where we used $\d=\e^{\frac{1}{m+2}} \leq r$ and the assumption that $N$ is $\kappa$-noncollapsed on the scale $r$ in the last inequality. So putting $\d=\e^{\frac{1}{m+2}}$ into the above yields that $|\s(x_0)| \leq (\Lambda+\kappa^{-1/2}) \e^{\frac{1}{m+2}}$, contradicting the assumption. This completes the proof.
\end{proof}
Now we go back to our situation, so let $L_t$ be the HLMCF in a hyperk\"ahler $4n$-manifold $M$. The above lemma indicates that it is important to study the evolution of the volume ratio along the flow.
\begin{lem}[Volume ratio estimate] \label{ncl}
If $L_0$ is $\kappa_0$-noncollapsed on the scale $r_0$, then for any small geodesic ball $B_t(x,\r)$ in $L_t$ with radius $\r \in (0,r_0)$, we have
\[
\Vol(B_t(x,\r)) \geq \kappa_0 e^{-(2n+1)E(t)} \r^{2n},
\]
where $E(t)$ is given by
\[
E(t):=\int_0^t (\max_{L_s}|H|^2+\max_{L_s}|A||H|)ds.
\]
\end{lem}
\begin{proof}
Let $\g_t$ be a length minimizing unit-speed geodesic w.r.t. $g(t)$ joining $p$ to $q \in B_t(p,\r)$. Then for every $t_0$ we have
\[
d_t(p,q)={\rm Length}_{g(t)}(\g_t) \leq {\rm Length}_{g(t)}(\g_{t_0}),
\]
and equality holds when $t=t_0$, which implies that
\[
\ddt d_t(p,q)|_{t=t_0}=\ddt {\rm Length}_{g(t)}(\g_t)|_{t=t_0}=\ddt {\rm Length}_{g(t)}(\g_{t_0})|_{t=t_0}.
\]
Thus we can compute
\[
\bigg| \ddt d_t(p,q) \bigg|=\bigg| \frac{1}{2} \int_0^{d_t(p,q)} \frac{d g_t}{dt} \bigg(\frac{d}{ds} \g_t, \frac{d}{ds} \g_t\bigg) ds \bigg| \leq \max_{L_t}|A||H| \cdot d_t(p,q).
\]
This implies that
\[
e^{-E(t)} d_0(p,q) \leq d_t(p,q) \leq d_0(p,q) e^{E(t)}, \quad d \mu_t \geq e^{-E(t)} d \mu_0.
\]
Since $L_0$ is $\kappa_0$-noncollapsed on the scale $r_0$, for $\r \leq r_0$, we have
\[
\Vol(B_t(p,\r)) =\int_{B_t(p,\r)} d \mu_t \geq \int_{B_0(p, e^{-E(t)} \r)} e^{-E(t)} d \mu_0 \geq \kappa_0 e^{-(2n+1)E(t)} \r^{2n}.
\]
The lemma is proved.
\end{proof}
\subsection{Some parabolic estimates for the HLMCF}
In this subsection, we prove some parabolic estimates for the HLMCF. The first lemma says that the HLMCF does not change a lot in short time intervals.
\begin{lem} \label{prs}
If $L_0$ satisfies
\[
|A|(0) \leq \Lambda, \quad |\na \Psi|(0) \leq P, \quad \lambda_1(0) \geq \d,
\]
then there exists $T=T(n,\Lambda,\bRm)$ such that the HLMCF $L_t$ satisfies
\[
|A|(0) \leq 2\Lambda,\quad |\na \Psi|(t) \leq 2P, \quad \lambda_1(t) \geq \frac{2}{3} \d, \quad t \in [0,T].
\]
\end{lem}
\begin{proof}
The estimate of $|A|$ follows from \cite[Lemma 2.2]{HS12}. Then the estimate of $\lambda_1$ follows from the exponential estimate for $\lambda_1$. Finally, we establish the estimate for $|\na \Psi|$. By the Bochner identity, Gauss equation and Proposition \ref{bja}, we can compute
\begin{align*}
	\bigg(\ddt-\Delta_t\bigg)|\na \Psi|^2 &=-2|\na^2\Psi|^2+\Rm^{\S^2} \ast (\na\Psi)^4
+\bRm \ast(\na\Psi)^2+A^2\ast(\na\Psi)^2\\
&\leq C(n, \Lambda,\bRm)|\na\Psi|^2. 
\end{align*}
Applying the maximum principle, we obtain
\[
|\na \Psi|(t) \leq e^{\frac{1}{2}C(n, \Lambda,\bRm)t}|\na \Psi|(0) \leq e^{\frac{1}{2}C(n, \Lambda,\bRm)t} P,
\]
so we may take $T \leq \frac{2\log 2}{C(n,\Lambda,\bRm)}$.
\end{proof}
We can obtain not only the usual smoothing estimates for $A$, but also for $\Psi$ with the help of Proposition \ref{bja}.
\begin{lem}[Smoothing estimates] \label{smg}
Suppose along the HLMCF, we have
\[
\sup_{L_t}|A| \leq \Lambda, \quad t \in [0,T]
\]
for some $T>0$. Then for each $l \geq 1$, there exist constants $\Lambda_l=\Lambda_l (n,\Lambda,\bRm, T)$ such that 
\[
\sup_{L_t}|\na^l A| \leq \frac{\Lambda_l}{t^{l/2}}, \quad t\in(0, T].  
\]
Moreover, for any $t_0\in(0, T]$, there exist constants $P_l=P_l(n, \Lambda, \bRm, t_0, T)$ such that 
\[
\sup_{L_t}|\na^l \Psi_\ast| \leq P_l, \quad t \in [t_0,T], 
\]
where $\Psi_*=\na\Psi$ is the differential map of the complex phase $\Psi: L \to \S^2$. 
\end{lem}
\begin{proof}
The estimate of $A$ follows from \cite[Theorem 3.1]{HS12}. Then for any $t_0\in (0, T]$ we have 
\[
\sup_{L_t}|\na^l A| \leq \frac{\Lambda_l}{(t_0/2)^{l/2}}, \quad t\in [t_0/2, T].  
\]
We use this estimate to show the estimate of $\Psi_\ast$. Note also that $|\Psi_\ast|$ has a uniform bound $|\Psi_\ast|\leq c(n)|A|\leq c(n)\Lambda$ by Proposition \ref{bja}.  

In order to derive the estimate of $\Psi_\ast$, we first compute the time derivative of $|\na^l\Psi_\ast|^2$ along the generalized harmonic map flow. A straight calculation shows that for each $l\geq 0$ we get the formula: 
\begin{align*}
	\ddt\na^l\Psi_\ast &= \Delta(\na^l \Psi_\ast)+\sum_{r+i+j+k=l}\tilde{\na}^r \Rm^{\S^2} \ast (\Psi_\ast)^r \ast \na^{i}\Psi_\ast \ast \na^{j}\Psi_\ast \ast \na^{k}\Psi_\ast\\
	&+\sum_{r+i_i+\cdots+i_l+j=l}\bar{\na}^r \bRm \ast\na^{i_1-1}A\ast \cdots \ast \na^{i_l-1}A\ast \na^j\Psi_\ast \\
	&+\sum_{i+j+k=l}\na^iA\ast \na^jA\ast \na^k\Psi_\ast, 
\end{align*}
where $\tilde{\na}$ denotes the Levi-Civita connection on $T\S^2$. 
It follows that for $t\in [t_0/2, T]$ we have 
\begin{align}\label{hde}
	\ddt|\na^l\Psi_\ast|^2 &= A^2\ast (\na^l\Psi_\ast)^2 + 2\Big\langle \ddt\na^l\Psi_\ast, \na^l \Psi_\ast\Big\rangle\nonumber \\
	& \leq \Delta |\na^l\Psi_\ast|^2-2|\na^{l+1}\Psi_\ast|^2 + C\sum_{0\leq i+j+k\leq l}|\na^i\Psi_\ast||\na^j\Psi_\ast||\na^k\Psi_\ast||\na^l\Psi_\ast|, 
\end{align}
where $C=C(n, \Lambda, \bRm, t_0, T)$ is a constant. From \eqref{hde} we have 
\[
\ddt|\Psi_\ast|^2\leq \Delta|\Psi_\ast|^2-2|\na \Psi_\ast|^2+c_1
\]
and 
\[
\ddt|\na\Psi_\ast|^2\leq \Delta|\na \Psi_\ast|^2-2|\na^2 \Psi_\ast|^2+c_2|\na\Psi_\ast|^2 + c_3,  
\]
where $c_k=c_k(n, \Lambda, \bRm, t_0, T) \; (k=1, 2, 3)$ are constants. 
Set 
\[
F:= (t-t_0/2)|\na \Psi_*|^2 + \a|\Psi_*|^2, 
\]
where $\a$ is a constant which will be determined later. It is not difficult to see 
\[
\Big(\ddt-\Delta\Big)F\leq (-2\a+1+Tc_2)|\na \Psi_\ast|^2+\a c_1+Tc_3. 
\]
Then we choose $\a=(1+TC_2)/2$ to get 
\[
\Big(\ddt-\Delta\Big)F \leq \Big(\frac{1+Tc_2}{2}\Big)c_1+Tc_3. 
\]
Applying the maximum principle, we have 
\[
F(t)\leq F(0)\leq \Big(\frac{1+Tc_2}{2}\Big)\Lambda^2=C_1(n, \Lambda, \bRm, t_0, T), \quad t\in [t_0/2, T]. 
\]
Hence we get 
\[
|\na\Psi_\ast|^2\leq \frac{C_1}{t-t_0/2}, \quad t\in (t_0/2, T]. 
\] 
It follows 
\[
\sup_{L_t}|\na \Psi_\ast|\leq\frac{6C_1}{t_0}=P_1(n, \Lambda, \bRm, t_0, T), \quad t \in [2t_0/3,T]. 
\]
This proves the case $l=1$. 

For $l\geq 2$, we prove it by induction. Assume that the following estimate holds for each $0\leq m \leq l-1$: 
\[
\sup_{L_t}|\na^m \Psi_\ast| \leq \frac{(m+1)(m+2)C_m(n, \Lambda, \bRm, t_0, T)}{t_0}, \quad t \in [((m+1)/(m+2))t_0,T]. 
\]
Then by \eqref{hde} we have 
\[
\ddt|\na^{l-1}\Psi_\ast|^2 \leq \Delta|\na^{l-1}\Psi_\ast|^2-2|\na^{l}\Psi_\ast|^2+c_4
\]
and 
\[
\ddt|\na^{l}\Psi_\ast|^2 \leq \Delta|\na^{l}\Psi_\ast|^2-2|\na^{l+1}\Psi_\ast|^2+c_5|\na^{l}\Psi_\ast|^2+c_6,  
\]
for $t \in [(l/(l+1))t_0,T]$, where $c_k=c_k(n, \Lambda, \bRm, t_0, T) \; (k=4, 5, 6)$ are constants which are controlled by the lower order estimates. By the same way as $l=1$, using maximum principle we see 
\[
|\na^{l}\Psi_\ast|^2\leq \frac{C_{l}(n, \Lambda, \bRm, t_0, T)}{t-(l/(l+1))t_0}, \quad t\in ((l/(l+1))t_0, T]. 
\]
Therefore we obtain the desired bound 
\[
|\na^{l}\Psi_\ast|^2\leq \frac{(l+1)(l+2)C_{l}(n, \Lambda, \bRm, t_0, T)}{t_0}=:P_l(n, \Lambda, \bRm, t_0, T)
\]
for $t\in [((l+1)/(l+2))t_0, T]$. 
\end{proof}
\begin{rk} \label{ues}
From the smoothing estimates, for any $t_0 \in (0,T)$ we have
\[
\sup_{L_t}|\na^l A| \leq \Lambda_l (n,\Lambda, \bRm, t_0), \quad \sup_{L_t}|\na^l \Psi| \leq P_l(n,\Lambda, \bRm, t_0), \quad t \in [t_0/2,t_0].
\]
In particular, we have bounds for the derivatives $|\na^l A|$ and $|\na^l \Psi|$ for $l \geq 1$ at $t=t_0$.
On the other hand, as in the proof of the above lemma, it is not difficult to see that we have bounds which only depend on $n$, $A(t_0)$ and $\Psi_\ast(t_0)$ (including their higher order derivatives)
\[
\sup_{L_t}|\na^l A| \leq \Lambda_l (n,A(t_0),\bRm), \quad \sup_{L_t}|\na^l \Psi| \leq P_l(n,A(t_0), \Psi_\ast(t_0),\bRm), \quad t \in [t_0,T].
\]
Combining the both estimates on $[t_0,T]$, we obtain $T$-independent estimates
\[
\sup_{L_t}|\na^l A| \leq \Lambda_l(n,\Lambda,\bRm, t_0), \quad
\sup_{L_t}|\na^l \Psi| \leq P_l(n,\Lambda, \bRm, t_0), \quad t \in [t_0,T].
\]
We often use this property without mentioning in later arguments.
\end{rk}
\subsection{Convergence of the flow}
Now we are ready to prove the main theorem.
\begin{thm}[Theorem \ref{coHLMCF}]
Let $(M, \bar{g})$ be a hyperk\"ahler $4n$-manifold with bounded geometry. Suppose $L$ is a hyper-Lagrangian submanifold with the complex phase $\Psi_0$ which is smoothly immersed into $M$. Then for any $V_0$, $\Lambda_0$ and $\d_0>0$, there exists $\e_0=\e_0(n,V_0,\Lambda_0,\d_0, \bRm, \inj(M))>0$ such that if $L$ satisfies
\[
\Vol(L_0) \leq V_0, \quad |A|(0) \leq \Lambda_0, \quad \lambda_1(\Delta_L)(0) \geq \d_0, \quad \cT(L_0) \leq \e_0,
\]
then the hyper-Lagrangian mean curvature flow starting from $L$ converges smoothly, exponentially fast to a complex Lagrangian submanifold in $M$ for one of the hyperk\"ahler complex structure on $M$.
\end{thm}
\begin{proof}
\textbf{Step 1. (Reduction from $L^2$ to $C^0$)}:
In the first step, we see that after a short period of time, the parabolicity of the flow improves the initial $L^2$-condition for $\na \Psi$ to the $C^0$-condition. From Proposition \ref{bja} and Lemma \ref{prs}, we know that $L_t$ satisfies
\[
|A|(t) \leq 2\Lambda_0, \quad |\na \Psi|(t) \leq c(n) \Lambda_0, \quad \lambda_1(t) \geq \frac{2}{3} \d_0, \quad t \in [0,T_0]
\]
for $T_0=T_0(n,\Lambda_0,\bRm)$. So Lemma \ref{efa} implies the following exponential estimate for the twistor energy:
\[
\cT(L_t) \leq e^{ct} \cT(L_0) \leq \e_0 e^{ct}, \quad t \in [0,T_0]
\]
for some $c=c(n,\Lambda_0)>0$. Therefore we can choose $t_0=t_0(n,\Lambda_0) \in (0, T_0]$ so that
\[
\cT(L_t) \leq 2 \e_0, \quad t \in [0,t_0].
\]
On the other hand, by the smoothing estimates, we know that for any $l \geq 1$,
\begin{equation} \label{ssa}
|\na^l A|(t) \leq C_l(n,\Lambda_0,\bRm), \quad t \in [ t_0/2, t_0 ],
\end{equation}
and also
\[
|\na^2 \Psi|(t) \leq c(n,\Lambda_0,\bRm), \quad t \in [ t_0/2, t_0 ].
\]
In order to get the estimate for the energy density $|\na \Psi|$, we need to establish the non-collapsing estimate for $L_t$ at first. By \cite[Proposition 2.2]{CH10} and \eqref{ssa}, we know that the injectivity radius of $L$ is bounded from below along the HLMCF
\[
\inj(L_t) \geq \iota(n,\Lambda_0,\bRm,\inj(M))>0, \quad t \in [t_0/2,t_0].
\]
Meanwhile, the Gauss equation implies that
\[
|\Rm| \leq C(\Lambda_0, \bRm), \quad t \in [t_0/2,t_0].
\]
So in the same way as the proof of \cite[Theorem 1.1]{Li12}, the volume comparison theorem yields that there exists $\kappa=\kappa(n,\Lambda_0,\bRm,\inj(M))$ and $r=r(n,\Lambda_0,\bRm,\inj(M))$ such that $L_t$ is $\kappa$-noncollapsed on the scale $r$ for all $t \in [t_0/2,t_0]$.
So Lemma \ref{efz} implies that
\[
|\na \Psi|(t) \leq (c+\kappa^{-1/2}) (2 \e_0)^{\frac{1}{2n+2}}=:\eta, \quad t \in [ t_0/2,t_0 ],
\]
where we take $\e_0$ sufficiently small so that $2 \e_0 \leq r^{2n+2}$.

\vspace{2mm}

\textbf{Step 2. ($\e_0$-regularity)}:
We set
\begin{align*}
\cA(\kappa,r,\Lambda,P,\d):=
\left\{
\begin{array}{l|l}
 &\; \text{$L$ is a hyper-Lagrangian submanifold}\\
 L\subset M\; &\; \text{$L$ is $\kappa$-noncollapsed on the scale $r$}\\
 & |A|\; \leq \Lambda, \quad |\na \Psi| \leq P, \quad \lambda_1(\Delta_L) \geq \d
\end{array}
\right\}. 
\end{align*}
Without loss of generality, we regard $L_{t_0/2}$ as the initial data of the HLMCF, so we have
\[
L_t \in \cA(\kappa,r,\Lambda,\eta,\d), \quad t \in [0,t_0/2],
\]
where $\Lambda:=2\Lambda_0$, $\eta:=(c+\kappa^{-1/2}) (2 \e_0)^{\frac{1}{2n+2}}$, $\d:=\frac{2}{3}\d_0$.
So Lemma \ref{prs} combining with the volume ratio estimate (\cf Lemma \ref{ncl}) implies that we can choose a small $T^{\ast}>0$ such that
\[
L_t \in \cA\bigg(\frac{1}{3}\kappa,r,6\Lambda,2\eta^{\frac{1}{2n+2}},\frac{1}{3}\d\bigg), \quad t \in [0,T^{\ast}].
\]
Let $T^{\ast}$ be the maximal time such that the above estimate holds. Then in order to prove the long-time existence of the flow, it suffices to prove the following $\e_0$-regularity:
\begin{clm}
There exists a small $\eta>0$ (and hence small $\e_0>0$) such that
\[
L_t \in \cA\bigg(\frac{2}{3}\kappa, r, 3\Lambda, \eta^{\frac{1}{2n+2}}, \frac{1}{2} \d\bigg), \quad t \in [0,T^{\ast}].
\]
\end{clm}
Indeed, if $T^{\ast}<\infty$ then from the claim we have $L_t \in \cA(\frac{2}{3}\kappa, r, 3\Lambda, \eta^{\frac{1}{2n+2}}, \frac{1}{2} \d)$ for $t \in [0,T^\ast]$. By using Lemma \ref{prs} and volume ratio estimate again, we find that there exists $\tilde{T}>T^{\ast}$ such that $L_t \in \cA(\frac{1}{3}\kappa,r,6\Lambda,2\eta^{\frac{1}{2n+2}},\frac{1}{3}\d)$ for $t \in [0,\tilde{T}]$, contradicting the maximality of $T^{\ast}$.

First, we establish an estimate for $|\na \Psi|$. We know that
\[
\lambda_1(t) \geq \frac{1}{3} \d, \quad t \in [0,T^{\ast}].
\]
So if we choose $\eta>0$ small so that
\[
\lambda_1(t) \geq \frac{1}{4} \d+C(n) \cdot 3 \Lambda \cdot 2 \eta^{\frac{1}{2n+2}}+(2\eta^{\frac{1}{2n+2}})^2, \quad t \in [0,T^{\ast}],
\]
then the exponential estimate for the twistor energy (\cf Lemma \ref{efa}) implies
\[
\cT(L_t) \leq e^{-\frac{\d}{2} t} \cT(L_0) \leq \eta^2 V_0 e^{-\frac{\d}{2} t}, \quad t \in [0,T^{\ast}].
\]
By Lemma \ref{prs}, there exists some $t^{\ast}=t^{\ast}(n,\Lambda,\bRm) \in (0,T^{\ast})$ such that
\[
|\na \Psi| \leq 2\eta \leq \eta^{\frac{1}{2n+2}}, \quad t \in [0,t^{\ast}],
\]
for $\eta \leq \frac{1}{2}$. On the other hand, since $|A|(t) \leq 6\Lambda$ for $t \in [0,T^{\ast}]$, the smoothing estimates imply that
\[
|\na^2 \Psi| \leq C(n,\Lambda, \bRm), \quad t \in [t^{\ast},T^{\ast}].
\]
Thus we obtain
\begin{equation} \label{edp}
|\na \Psi|(t) \leq C(n,\Lambda,\kappa,r,V_0,\bRm) \cdot \eta^{\frac{1}{n+1}} e^{-\frac{\d t}{4n+4}}, \quad t \in [t^\ast,T^{\ast}].
\end{equation}
So we can choose $\eta>0$ small so that
\[
C(n,\Lambda,\kappa,r,V_0,\bRm) \cdot \eta^{\frac{1}{2n+2}} \leq 1
\]
and obtain
\[
|\na \Psi|(t) \leq \eta^{\frac{1}{2n+2}}, \quad t \in [0,T^{\ast}].
\]
Next, we compute $|A|$. By the smoothing estimates, for any $l \geq 1$, we have
\[
|\na^l A| \leq C_l(n,\Lambda,\bRm), \quad t \in [t^{\ast}, T^{\ast}].
\]
Thus we also have
\[
|\na^l H| \leq C_l(n,\Lambda,\bRm), \quad t \in [t^{\ast}, T^{\ast}].
\]
From Proposition \ref{hdp} and \eqref{edp}, we know that $|H|$ also decreases exponentially fast. So integrating by parts, we have
\[
\int_{L_t} |\na^2 H|^2 d \mu_t \leq \int_{L_t} |H| |\na^4 H| d \mu_t \leq C(n,\Lambda,\kappa,r,V_0,\bRm) \eta^{\frac{1}{n+1}} e^{-\frac{\d t}{4n+4}}
\]
for $t \in [t^\ast,T^{\ast}]$. So we have
\[
|\na^2 H| \leq c(n,\Lambda,\kappa,r,V_0,\bRm) \eta^{\frac{1}{2(n+1)^2}} e^{-\frac{\d t}{8(n+1)^2}}, \quad t \in [t^{\ast}, T^{\ast}].
\]
We recall the evolution equation of $A$ along the MCF (\cf \cite{CL01})
\begin{align*}
	\ddt h^{\a}_{ij}=\na_i\na_j H^\a-H^{\b}h^{\b}_{jk}h^{\a}_{ik}+H^{\b}\bar{R}_{\a j \b i}+h^{\b}_{ij}b^{\b}_{\a}, 
\end{align*}
where $b^{\b}_{\a}=\bar{g}(\ddt e_{\a}, e_\b)=\bar{g}(\bar{\na}_{H}e_\a, e_\b)$.
Note that $b^{\b}_{\a}$ is anti-symmetric since  
\[
0=\ddt(\bar{g}(e_\a, e_\b))=b_{\a}^\b+b_{\b}^{\a}. 
\]
Then it follows 
\begin{align*}
	h^{\a}_{ij}h^{\b}_{ij}b^{\a}_{\b}=0. 
\end{align*}
So we compute 
\begin{align*}
	2|A|\ddt|A|=\ddt|A|^2\leq c(n)(|\na^2 H||A|+|H||A||\bRm|+|H||A|^3).
\end{align*}
Dividing  both sides by $|A|$, we have
\begin{align} \label{efs}
	\ddt|A|\leq c(n)(|\na^2 H|+|H||\bRm|+|H||A|^2). 
\end{align}
Meanwhile, Lemma \ref{prs} shows that
\[
|A|(t) \leq 2\Lambda, \quad t \in [0,t^{\ast}].
\]
So integrating \eqref{efs} in $t$ and using the exponential decay of $|H|$, we have
\begin{eqnarray*}
|A|(t) &\leq& |A|(t^{\ast})+c(n) \int_{t^{\ast}}^t(|\na^2 H|+|H||\bRm|+|H||A|^2)ds \\
&\leq& 2\Lambda+c(n) \bigg[ c \eta^{\frac{1}{2(n+1)^2}} \frac{16(n+1)^2}{\d}+(C(\bRm)+64\Lambda^2) \cdot c \eta^{\frac{1}{n+1}} \frac{8(n+1)}{\d} \bigg].
\end{eqnarray*}
Thus we can take $\eta>0$ sufficiently small so that
\[
|A|(t) \leq 3\Lambda, \quad t \in [0,T^{\ast}].
\]
Then we establish the estimate for $\lambda_1(t)$. Since $\lambda_1(0) \geq \d$, Lemma \ref{prs} shows that
\[
\lambda_1(t) \geq \frac{2}{3}\d, \quad t \in [0,t^{\ast}].
\]
Thus the exponential estimate for $\lambda_1$ combining with the exponential decay of $|H|$ imply that
\begin{eqnarray*}
\lambda_1(t) &\geq& \exp \bigg[-\int_{t^\ast}^t(\max_{L_s} |H|^2+C(n) \max_{L_s} |H||A|)ds \bigg] \lambda_1(t^\ast) \\
&\geq& \exp \bigg[-c^2 \eta^{\frac{2}{n+1}} \frac{4(n+1)}{\d}-C(n) \cdot 3\Lambda \cdot c \eta^{\frac{1}{n+1}} \frac{8(n+1)}{\d} \bigg] \lambda_1(t^\ast).
\end{eqnarray*}
If we take $\eta>0$ sufficiently small, then
\[
\lambda_1(t) \geq \frac{1}{2} \d, \quad t \in [0,T^{\ast}].
\]
We can prove a non-collapsing estimate of $L_t$ in the same way as $\lambda_1$, by using the volume ratio estimate.

\vspace{2mm}

\textbf{Step 3. (Exponential convergence of the flow)}:
From Step 2, we have a uniform bound for $A$. So the standard bootstrapping arguments combining with Simon's theorem \cite{Sim83} imply the smooth convergence of the MCF $L_t \to L_{\infty}$. Moreover, we have already seen that for a fixed sufficiently small $\eta>0$, we have
\[
|\na \Psi(t)| \leq C(n,\Lambda,\kappa,r,V_0,\bRm) \cdot \eta^{\frac{1}{n+1}} e^{-\frac{\d t}{4n+4}} \searrow 0.
\]
In particular, Proposition \ref{hdp} implies that $H_t$ converges exponentially fast to $H_{\infty}=0$, and hence $L_{\infty}$ is minimal.

As for the generalized harmonic map flow, we have also the uniform bounds $|\na^l \Psi| \leq C_l$ for all $l \geq 1$. Thus there exists a subsequence $\{\Psi_{t_i}\}$ which converges to a smooth map $\Psi_{\infty} \colon L \to \S^2$ and $L_{\infty}$ inherits a hyper-Lagrangian structure with the complex phase $\Psi_{\infty}$. Since $|\na \Psi_{\infty}|=0$, the map $\Psi_{\infty}$ should be a constant. Finally, we show that the complex phase $\Psi_{\infty}$ which arises from the generalized harmonic map flow does not depend on the choice of the subsequence $\{\Psi_{t_i}\}$ by contradiction. So we assume that there exist two distinct constant phase maps $\Psi_{\infty}$ and $\Psi'_{\infty}$ which arise in this way. We take a small geodesic ball in $B \subset \S^2$ centered at $\Psi_{\infty}$ so that $\Psi'_{\infty} \not\in B$. Since $\{\Psi_{t_i}\}$ converges to $\Psi_{\infty}$ we know that $\Psi_{t_i}(L) \subset B$ for $i$ large enough. We fix such an $i$ and consider the generalized harmonic map flow $\Psi'_t$ starting from the data $(L_{t_i},\Psi_{t_i})$. Then a simple maximum principle argument (\cf Corollary \ref{mpl}) shows that $\Psi'_t(L) \subset B$ for all $t \in [0,\infty)$ whereas $\{\Psi'_t\}$ should have a convergent subsequence to $\Psi'_{\infty} \not\in B$, so contradiction. This completes the proof.
\end{proof}

\end{document}